\documentclass[a4paper, 11pt]{amsart}
\usepackage{amsmath}
\usepackage{amssymb}
\usepackage{amsfonts}
\usepackage{amsthm}
\usepackage{ifpdf}
\usepackage{color}
\usepackage[dvipdfm]{graphicx}
\usepackage{fancyhdr}

\theoremstyle{plane}
	\newtheorem{thm}{Theorem}[section]
	\newtheorem{lem}{Lemma}[section]
	\newtheorem{prop}{Proposition}[section]
	\newtheorem{cor}{Corollary}[section]
	
\theoremstyle{definition}
	
	\newtheorem{exmp}{Example}[section]
	\newtheorem{rmk}{Remark}[section]

\newcommand{\Z}{\mathbb Z}
	
\newcommand{\RP}{{\mathbb R}{\mathbb P}}


\begin{document}

\title{Construction of spines of two-bridge link complements and upper bounds of their Matveev complexities}
\author{Masaharu Ishikawa and Keisuke Nemoto}
\date{}

\maketitle
\thispagestyle{empty}

\begin{abstract}

We give upper bounds of the Matveev complexities of two-bridge link complements by constructing their spines explicitly. In particular, we determine the complexities for an infinite sequence of two-bridge links corresponding to the continued fractions of the form $[2,1,\ldots,1,2]$. We also give upper bounds for the 3-manifolds obtained as meridian-cyclic branched coverings of the 3-sphere along two-bridge links.

\end{abstract}

\section{Introduction}

Let $M$ be a compact connected 3-manifold possibly with boundary. If $M$ has nonempty boundary then a  polyhedron $P\subset M$ to which $M$ collapses is called a {\it spine} of $M$. If $M$ is closed then a spine of $M$ means that of $M\setminus B^3$, where $B^3$ is a 3-ball in $M$. A spine $P$ of $M$ is said to be {\it almost-simple} if the link of any point can be embedded into the complete graph $K_4$ with four vertices. A point of almost-simple spine whose link is $K_4$ is called a {\it true vertex}. The minimal number $c(M)$ of true vertices among all almost-simple spines of $M$ is called the {\it complexity} of $M$.

The notion of the complexity was introduced by S. Matveev in \cite{Matveev2}. The complexity gives an efficient measure on the set of all compact 3-manifolds $\mathcal{M}$, because it has the following properties: the complexity is additive under connected sum, that is, $c(M_1\sharp M_2)=c(M_1)+c(M_2)$, and it has a finiteness property, that is, for any $n\in \Z_{\geqslant 0}$, there exists finitely many closed irreducible manifolds $M\in \mathcal{M}$ with $c(M)=n$. Note that if $M$ is closed, irreducible and other than $S^3$, $\RP^3$ and $L(3,1)$ then $c(M)$ coincides with the minimal number of ideal tetrahedra of all triangulations of $M$.

Determining the complexity $c(M)$ of a given 3-manifold $M$ is very difficult in general. For the complexity of the lens space $L(p,q)$, Matveev proved the upper inequality $c(L(p,q))\leqslant S(p,q)-3$, where $S(p,q)$ is the sum of all partial quotients in the expansion of $p/q$ as a regular continued fraction with positive entries, and conjectured that the equality holds (see also \cite{Matveev}). In recent studies, Jaco, Rubinstein and Tillmann solved this conjecture positively for some infinite sequences of lens spaces \cite{Jaco}. Petronio and Vesnin studied the complexity of closed 3-manifolds which are obtained as meridian-cyclic branched coverings of $S^3$ along two-bridge links \cite{PV}. In the case of compact manifolds with nonempty boundary, Fominykh and Wiest obtained sharp upper bounds on the complexity of torus link complements \cite{Fominykh}. A certain lower bound of the complexity of a two-bridge link complement is given in \cite{PV}. There are several related studies, see for instance \cite{Frigerio,Frigerio2,Anisov,Jaco2,Yu,Yu2}.

In this paper, we give upper bounds of the complexities of two-bridge link complements.

Let $K(p,q)$ be a two-bridge link in the 3-sphere $S^3$, where $p,q$ are coprime integers with $p\geqslant 2$ and $q \ne 0$. We may represent it by using Conway's notation as $C(a_1,\ldots ,a_n)$, where the integers $a_i$ are the partial quotients of a regular continued fraction of $p/q$. We represent the regular continued fraction of $p/q$ as $p/q=[a_1,\ldots , a_n]$. For each continued fraction, $K(p,q)$ has a diagram as shown in Figure \ref{fig1}. By taking the mirror image if necessary, we may assume that $p/q>0$. In this paper we only consider the continued fraction of $p/q$ such that each $a_i$ is positive and $a_1, \ a_n >1$. Let $N(K(p,q))$ be a compact tubular neighborhood of $K(p,q)$ in $S^3$ and ${\rm int}N(K(p,q))$ its interior.
%Set $l(p,q)=n$, which is called the {\it minimal length} of $p/q$ in \ref{PV}.
%We define $l(p,q)$ to be $k$ if $a_1>$ and $k-1$ if $a_1=1$ for a continued fraction of $p/q$ with positive entries, and it said to be a {\it minimal length} of $p/q$. 

%let $p/q$ be represented by $p/q=[a_1,\ldots , a_n]$, where the integers $a_i$ are the partial quotients of continued fraction. If each partial quotient corresponds to the crossing number of each tangle as shown in Figure \ref{fig1}, we denote by $K(p,q)$ or $C(a_1,\ldots ,a_n)$ the {\it two-bridge link (or knot)} in 3-sphere $S^3$. We define $l(p,q)$ to be $k$ if $a_1>$ and $k-1$ if $a_1=1$ for a continued fraction of $p/q$ with positive entries, and it said to be a {\it minimum length} of $p/q$. It is well-known that we can deform $K(p,q)$ into $C(a_1,\ldots , a_n)$ such that all $a_i$ are either positive or negative and neither $|a_1|$ nor $|a_n|$ is equal to one.  In addition $K(p,-q)$ is the mirror image of $K(p,q)$. Since we will not care about orientation, we can assume $a_i>0$ for all $i$ and $a_1, \ a_n >1$. \\

\begin{figure}[htbp]
\begin{center}
\includegraphics[scale=0.5]{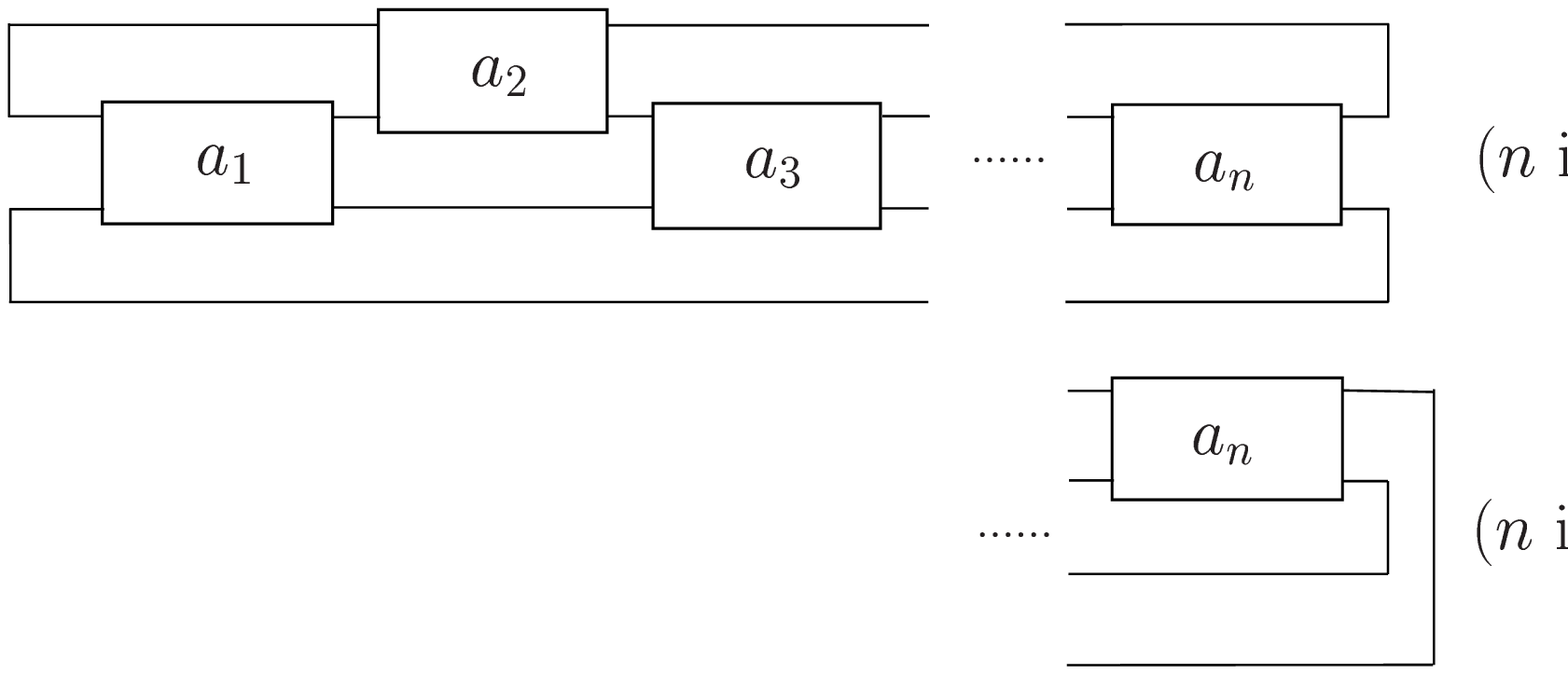}%\label{fig}
\end{center}
\begin{center}
\includegraphics[scale=0.9]{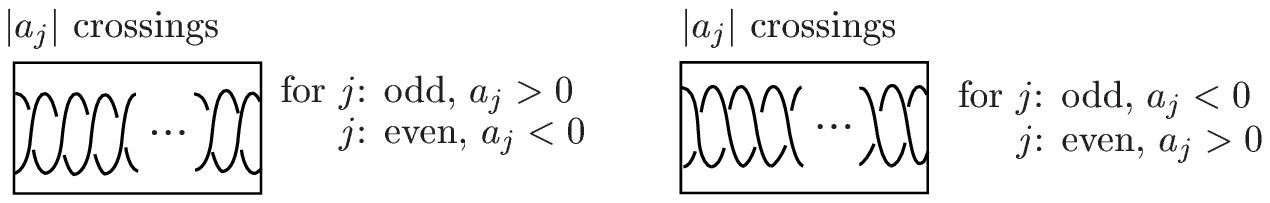}
\caption{Two-bridge link $C(a_1,\ldots ,a_n)$.}
\label{fig1}
\end{center}
\end{figure}

The main result of this paper is the following:

\begin{thm}\label{thm1}
Let $K(p,q)$ be a two-bridge link with $p/q=[a_1,\ldots ,a_n]$, $a_i>0$ and $a_1, \ a_n >1$. Then, 
\[
c(S^3\setminus \mathrm{int}N(K(p,q)))\leqslant \sum_{i=1}^n a_i +2(n-3)-\sharp \{  a_i=1 \},
\]
where $\sharp \{  a_i=1 \}$ is the number of indices $i$ such that $a_i=1$.  
\end{thm}
As a corollary, we determine the complexities of an infinite sequence of two-bridge links.

\begin{cor}\label{cor2}
Let $K(p,q)$ be a two-bridge link with $p/q=[2,\underbrace{1,\ldots ,1}_{n-2},2]$, $n\geqslant 2$. Then, 
\[
c(S^3\setminus \mathrm{int}N(K(p,q)))=2n-2.
\]

\end{cor}

As an application of Theorem \ref{thm1}, we obtain the following upper bounds on the
complexity of meridian-cyclic $d$-fold branched coverings of $S^3$, which is sharper than the upper bounds given in \cite{PV}.

\begin{cor}\label{cor3}
Let $M_d(K(p,q))$ be the meridian-cyclic branched covering of $S^3$ along $K(p,q)$ of degree $d$. Then,
\[
c(M_d(K(p,q)))\leqslant d\Bigl( \sum_{i=1}^n a_i +2(n-3)-\sharp \{ a_i=1\}\Bigr)+rd,
\] 
where, $r$ is one if $p$ is odd and three if $p$ is even.

\end{cor}

\vspace{2mm}

We are deeply grateful to Yuya Koda for valuable information and instructive discussion on the results. The first author is supported by the Grant-in-Aid for Scientific Research (C), JSPS KAKENHI Grant Number 25400078.

\section{Proof of Theorem \ref{thm1}}

We will show the upper bound of the complexity in Theorem \ref{thm1} by constructing an almost-simple spine of the two-bridge link complement explicitly.

Let $K(p,q)$ be a two-bridge link in the 3-sphere $S^3$, where $p,q$ are coprime integers with $p\geqslant 2$ and $q > 0$. We may represent it as $C(a_1,\ldots ,a_n)$, where the integers $a_i$ are the partial quotients of a regular continued fraction of $p/q$. For each $i=1,\ldots, n$, let $T_i$ be the tangle containing the $a_i$ twists in Figure \ref{fig1}, which is a 3-ball with two tunnels mutually twisted $a_i$-times. For each $i=1,\ldots ,n$, let $A_i$ be the union of a sphere with four holes and a disk which is placed inside the sphere twisted $a_i$-times as shown in Figure \ref{fig4}. Each $A_i$ is a spine of $T_i$ obtained by collapsing it from $\partial N(K(p,q))$ with keeping the holed sphere on the boundary of $T_i$. In this paper, we call $A_i$ a {\it pillowcase}. We denote the disk lying in the middle of $A_i$ by $D_i$. The pillowcase $A_i$ has four boundary components. We denote these components by $\partial A_i^{{\rm NW}}$, $\partial A_i^{{\rm NE}}$, $\partial A_i^{{\rm SW}}$, $\partial A_i^{{\rm SE}}$, see Figure \ref{fig4}. We will make a spine of $S^3\setminus K(p,q)$ by gluing these pillowcases.

\begin{figure}
\begin{center}
\hspace{1.5cm}
\includegraphics[scale=0.5]{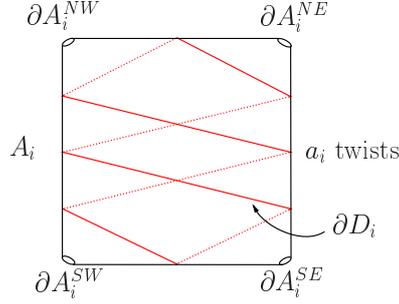}
\caption{Pillowcase $A_i$ and its boundary components.}
\label{fig4}
\end{center}
\end{figure}

We first prove the following weaker version of Theorem \ref{thm1}.

\begin{lem}\label{lem1}
Let $C(a_1,\ldots ,a_n)$ be a two-bridge link, where each $a_i$ is positive and $a_1, \ a_n >1$. Then, 
\[
c(S^3\setminus \mathrm{int}N(C(a_1,\ldots ,a_n)))\leqslant \sum_{i=1}^n a_i +2(n-3).
\]
\end{lem}

\begin{proof}
We first construct a spine $P$ of $S^3\setminus C(a_1,\ldots ,a_n)$ by gluing the pillowcases $A_1,\ldots ,A_n$ together as follows:
\begin{itemize}
\item For each adjacent pair of pillowcases $A_i$ and $A_j$, $1\leqslant i,j\leqslant n-1$,
we attach two tubes in order to pass the knot strands and then glue a disk for each region bounded by tubes and pillowcases as shown in Figure \ref{fig5}.
\item We attach three tubes and two disks to the pillowcase $A_n$ as shown in Figure \ref{fig6}. 
\end{itemize}
\begin{figure}[htbp]
\begin{center}
\includegraphics[scale=0.8]{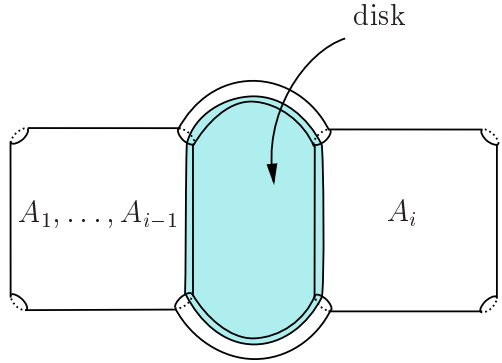}
\caption{Gluing a pair of adjacent pillowcases. }
\label{fig5}
\end{center}
\end{figure}

\begin{figure}[htbp]
\begin{center}
\includegraphics[scale=0.8]{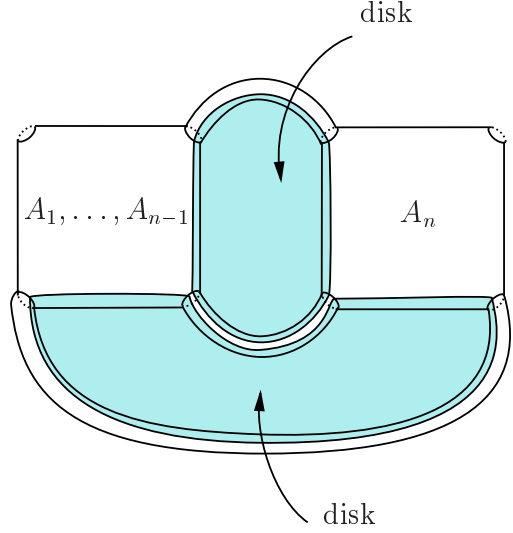}
\caption{Gluing $A_{n-1}$ and $A_n$.}
\label{fig6}
\end{center}
\end{figure}

Next, we make a spine $P_0$ by collapsing $P$ from $\partial A_1$ and $\partial A_n$, which decreases the number of true vertices. Let $x_i$ be the number of true vertices on $\partial D_i$ in the spine $P_0$. Since any true vertex of $P$ lies on $\partial D_i$ for some $1\leqslant i \leqslant n$, there exists $\sum_{i=1}^n x_i $ true vertices on $P_0$.

We calculate the number of true vertices in the spine $P_0$.

\begin{itemize}
\item For the pillowcase $A_1$.\\
True vertices which lie on $\partial D_1$ are $y_1^{(1)},\ldots ,y_{a_{1}}^{(1)},y_{a_{1}+1}^{(1)}$ shown in Figure \ref{fig7}. Since the true vertices $y_{1}^{(1)},y_{2}^{(1)}$ are removed by the collapsing from $\partial A_{1}^{\text{NW}}$, we get $x_{1}=a_{1}-1$. 

\begin{figure}[htbp]
\begin{center}
\includegraphics[scale=1]{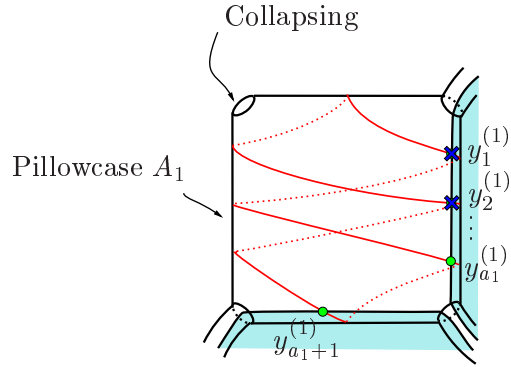}
\caption{True vertices in the pillowcase $A_1$.}
\label{fig7}
\end{center}
\end{figure}

\item For pillowcases $A_i$, where $i=2,\ldots , n-1$.\\
True vertices which lie on $\partial D_i$ are $y_1^{(i)},\ldots ,y_{a_{i}}^{(i)},y_{a_{i}+1}^{(i)},y_{a_{i}+2}^{(i)}$ shown in Figure \ref{fig8}. Therefore, we get $x_{i}=a_{i}+2$.
\begin{figure}[htbp]
\begin{center}
\includegraphics[scale=1]{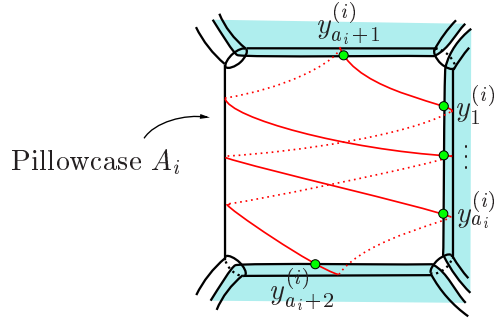}
\caption{True vertices in the pillowcase $A_i$.}
\label{fig8}
\end{center}
\end{figure}

\item For the pillowcase $A_n$.\\
True vertices which lie on $\partial D_n$ are $y_1^{(n)},\ldots ,y_{a_{n}}^{(n)},y_{a_{n}+1}^{(n)}$ shown in Figure \ref{fig9}. Since the true vertices $y_{a_{n}-1}^{(n)},y_{a_{n}}^{(n)}$ are removed by the collapsing from $\partial A_{n}^{\text{SW}}$, we get $x_{n}=a_{n}-1$. 

\begin{figure}[htbp]
\begin{center}
\hspace{0.5cm}
\includegraphics[scale=1]{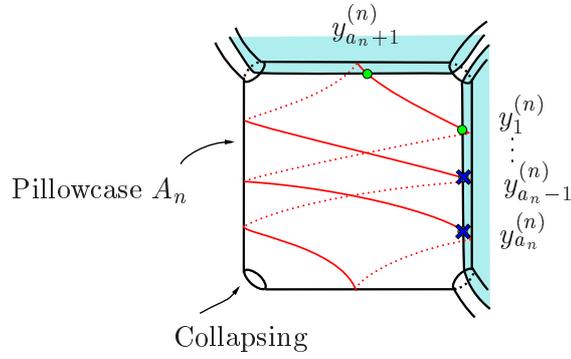}
\caption{True vertices in the pillowcase $A_n$.}
\label{fig9}
\end{center}
\end{figure}
\end{itemize}

\noindent In summary, the number of true vertices in each pillowcase is
\[
\begin{cases}
\begin{split}
x_{1}&= a_{1}-1\\
x_{i}&=a_{i}+2 \;\;\;(i=2,\ldots ,n-1) \\
x_{n}&=a_{n}-1.
\end{split}
\end{cases}
\]
Therefore, the number of true vertices in the spine $P_0$ is
\[
\begin{split}
\displaystyle\sum_{i=1}^n x_i &= a_1-1+\sum_{i=2}^{n-1}(a_i+2)+a_n-1\\
&= \sum_{i=1}^n a_i +2(n-3).
\end{split}
\]
\end{proof}

\begin{proof}[Proof of Theorem \ref{thm1}]
We will make a new spine $P'$ from $P_0$ constucted in Lemma \ref{lem1} by collapsing it as follows (An example of $P'$ is given in Example \ref{ex} below) : Let $1\leqslant i_1<i_2<\cdots <i_r\leqslant n$ be the set of indices with $a_{i_{j}}=1$, $j=0,\ldots ,r-1$. For each $j$, let $P_{j+1}$ be a spine obtained from $P_j$ by applying the replacement shown in Figure \ref{fig11}. The left figure represents the replacement in the case of $a_{i_j-1}>1$ and the right one is in the case of $a_{i_j-1}=1$. Applying this replacement inductively for $j=0,\ldots ,r-1$, we get a new spine $P_r$ of $S^3\setminus C(a_1,\ldots ,a_n)$.

\begin{figure}[htbp]
 \begin{minipage}{0.49\hsize}
  \begin{center}
  \includegraphics[scale=0.6]{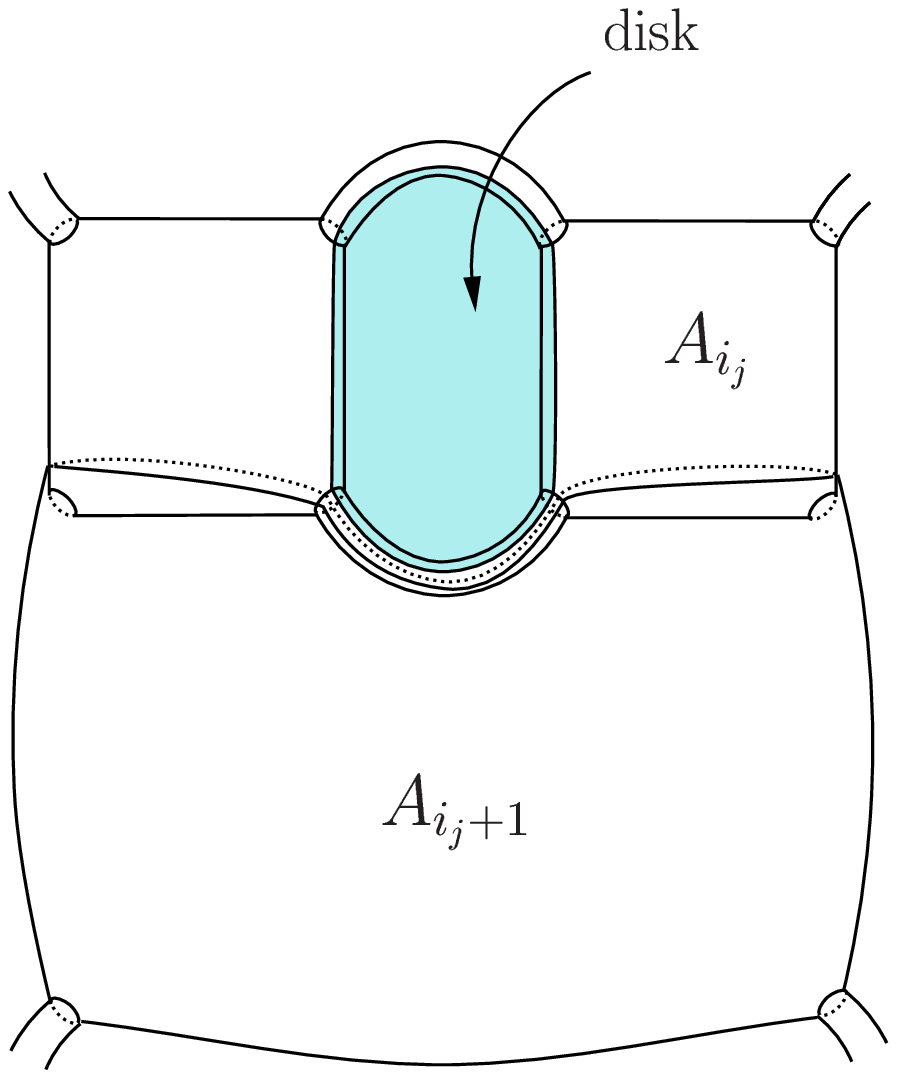}
  \end{center} 
\end{minipage}
 \begin{minipage}{0.49\hsize}
\begin{center}
  \includegraphics[scale=0.6]{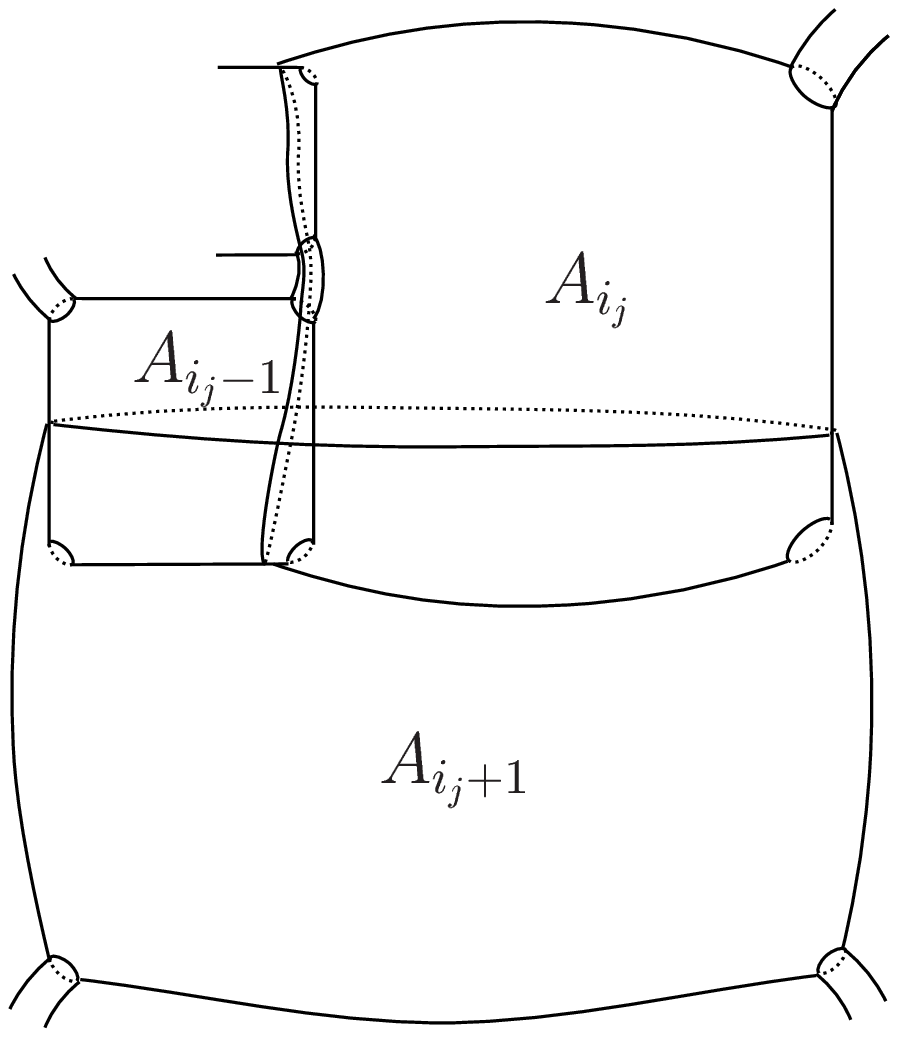}
\end{center}     
\end{minipage}
  \caption{Gluing pillowcases together in the case of $a_{i_j-1}>1$ (shown in the left hand side) and $a_{i_j-1}=1$ (shown in the right hand side).}\label{fig11}
\end{figure}

Let $P'_j$ be the spine obtained from $P_j$ by collapsing it from the boundary. In the following we fix $j$, $0\leqslant j\leqslant r$, and denote $i_j$ by $i$ for simplicity.  Let $x_k^{(j)}$ be the number of true vertices on $\partial D_k$ in the spine $P'_j$, where $k=1,\ldots ,n$. All true vertices of $P_j$ which do not lie on $\partial D_k$ are removed in $P'_j$ by the collapsing. Therefore, there exists $\sum_{k=1}^n x_k^{(j)} $ true vertices in $P'_j$.

We now calculate the difference of the numbers of true vertices between $P'_j$ and $P'_{j+1}$.

\begin{itemize}
\item[(i)] In the case of $a_{i-1}>1$.
	
	\item For the pillowcase $A_{i-1}$.\\
	True vertices which lie on $\partial D_{i-1}$ in $P'_j$ are $y_1^{(i-1)},\ldots ,y_{a_{i-1}}^{(i-1)},y_{a_{i-1}+1}^{(i-1)},y_{a_{i-1}+2}^{(i-1)}$ shown on the left hand side in Figure \ref{fig77}. Hence, we get $x_{i-1}^{(j)}=a_{i-1}+2$. On the other hand, in $P'_{j+1}$, true vertices which lie on $\partial D_{i-1}$  are $z_1^{(i-1)},\ldots ,z_{a_{i-1}}^{(i-1)},z_{a_{i-1}+1}^{(i-1)},z_{a_{i-1}+2}^{(i-1)},z_{a_{i-1}+3}^{(i-1)}$ shown on the right hand side, and hence $x_{i-1}^{(j+1)}=a_{i-1}+3$. Thus $x_{i-1}^{(j+1)}=x_{i-1}^{(j)}+1$. 

	\item For the pillowcase $A_i$.\\
	True vertices which lie on $\partial D_i$ in $P'_j$ are $y_{a_i}^{(i)},y_{a_i+1}^{(i)},y_{a_i+2}^{(i)}$ shown on the left hand side in Figure \ref{fig77}, and hence we get $x_i^{(j)}=a_i+2=3$. On the other hand, true vertices which lie on $\partial D_i$ in the $P_{j+1}$ are $z_1^{(i)},\ldots ,z_4^{(i)}$ shown on the right hand side. By collapsing from $\partial A_i^{\text{SE}}$, the true vertices $z_1^{(i)},z_2^{(i)}$ are removed in $P'_{j+1}$. Hence $x_i^{(j+1)}=2$. Therefore we get $x_i^{(j+1)}=x_i^{(j)}-1$. 

	\item For the pillowcase $A_{i+1}$.\\
	True vertices which lie on $\partial D_{i+1}$ in $P'_j$ are 
$y_1^{(i+1)},\ldots  ,y_{a_{i+1}}^{(i+1)},y_{a_{i+1}+1}^{(i+1)},y_{a_{i+1}+2}^{(i+1)}$ 
shown on the left hand side in Figure \ref{fig77}, and hence we get $x_{i+1}^{(j)}=a_{i+1}+2$. 
On the other hand, true vertices which lie on $\partial D_{i+1}$ in $P_{j+1}$ are 
$z_1^{(i+1)},\ldots ,z_{a_{i+1}}^{(i+1)},z_{a_{i+1}+1}^{(i+1)}, z_{a_{i+1}+2}^{(i+1)},z_{a_{i+1}+3}^{(i+1)}$ 
shown on the right hand side. By collapsing from $\partial A_{i-1}^{\text{SW}}$, 
the true vertices $z_{a_{i+1}+2}^{(i+1)},z_{a_{i+1}+3}^{(i+1)}$ are removed in $P'_{j+1}$. 
Hence $x_{i+1}^{(j+1)}=a_{i+1}+1$. Thus we get $x_{i+1}^{(j+1)}=x_{i+1}^{(j)}-1$. 

	\item For the other pillowcases $A_k$, where $k\neq i-1,i,i+1$.\\
	Since the replacement $P_j\to P_{j+1}$ does not change the true vertices in $A_k$, we get $x_k^{(j+1)}=x_k^{(j)}$.

\begin{figure}[htbp]
 \begin{minipage}{0.49\hsize}
  \begin{center}
   \includegraphics[scale=0.9]{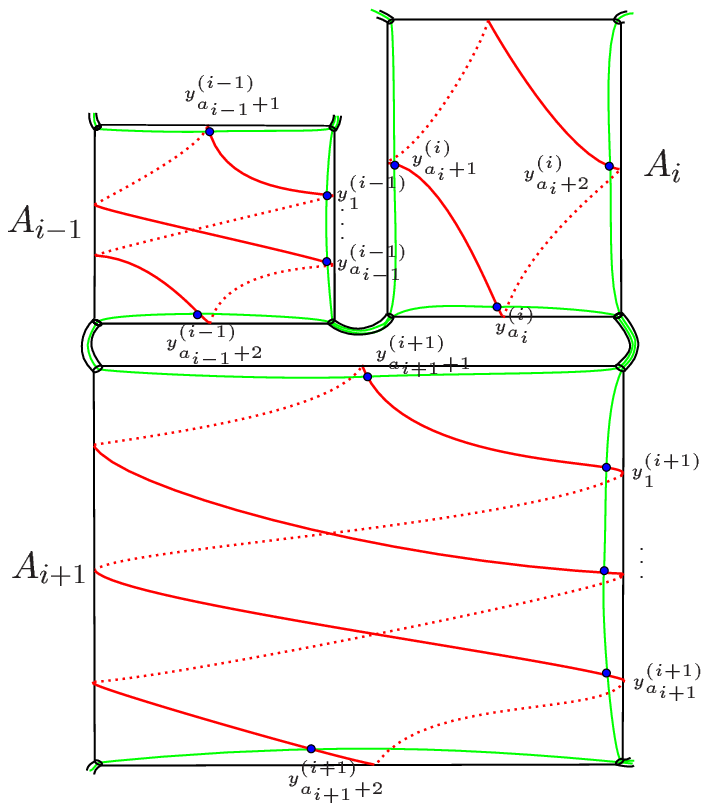}\label{fig76}
  \end{center}
 \end{minipage}
 \begin{minipage}{0.49\hsize}
  \begin{center}
   \includegraphics[scale=0.9]{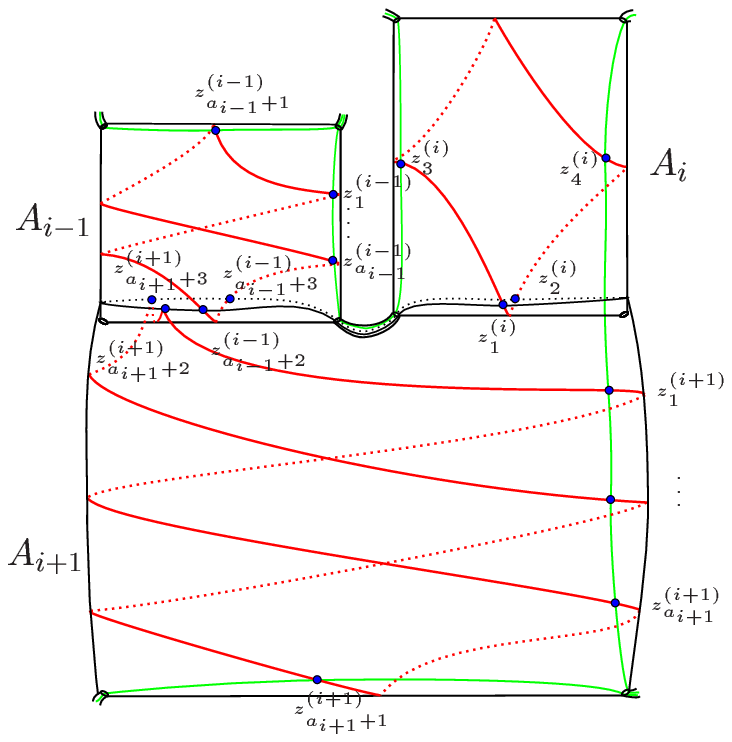}\label{fig77}
  \end{center}
  \end{minipage}
  \caption{True vertices in the spine $P_j$  (shown in the left hand side) and in the spine $P_{j+1}$ (shown in the right hand side) in the case of $a_{i-1}>1$.}
  \label{fig77}
\end{figure}

\item[(ii)] In the case of $a_{i-1}=1$.
	
	\item For the pillowcase $A_{i-1}$.\\
	By the replacement $P_j\to P_{j+1}$, the true vertex $z^{(i-1)}$ shown in Figure \ref{fig79} appears and this is not removed in $P'_{j+1}$. Therefore, we get $x_{i-1}^{(j+1)}=x_{i-1}^{(j)}+1$.  

	\item For the pillowcase $A_i$.\\
	True vertices which lie on $\partial D_i$ in $P_j$ are $y_{a_i}^{(i)},y_{a_i+1}^{(i)},y_{a_i+2}^{(i)},y_{a_i+3}^{(i)}$ shown on the left hand side in Figure \ref{fig79}. By collapsing $P_j$ from $\partial A_{i-2}^{\text{NE}}$, the true vertices $y_{a_i+2}^{(i)},y_{a_i+3}^{(i)}$ are removed in $P'_{j+1}$. Hence $x_i^{(j)}=a_i+1=2$. On the other hand, true vertices which lie on $\partial D_i$ in $P_{j+1}$ are $z_{a_i}^{(i)},\ldots ,z_{a_i+4}^{(i)}$ shown on the right hand side. By collapsing from $\partial A_{i-2}^{\text{NE}}$ and $\partial A_{i}^{\text{SE}}$, the true vertices $z_{a_i+2}^{(i)},z_{a_i+3}^{(i)}$ and $z_{a_i}^{(i)},z_{a_i+4}^{(i)}$ are removed in $P'_{j+1}$, respectively, and hence $x_i^{(j+1)}=1$. Thus we get $x_i^{(j+1)}=x_i^{(j)}-1$. 

	\item For the pillowcase $A_{i+1}$.\\
	Applying the same argument as in case (i), we get $x_{i+1}^{(j+1)}=x_{i+1}^{(j)}-1$.
	
	\item For the other pillowcases $A_k$, where $k\neq i-1,i,i+1$.\\
	Since the replacement $P_j\to P_{j+1}$ does not change the true vertices in $A_k$, we get $x_k^{(j+1)}=x_k^{(j)}$.

\begin{figure}[htbp]
 \begin{minipage}{0.49\hsize}
  \begin{center}
   \includegraphics[scale=0.9]{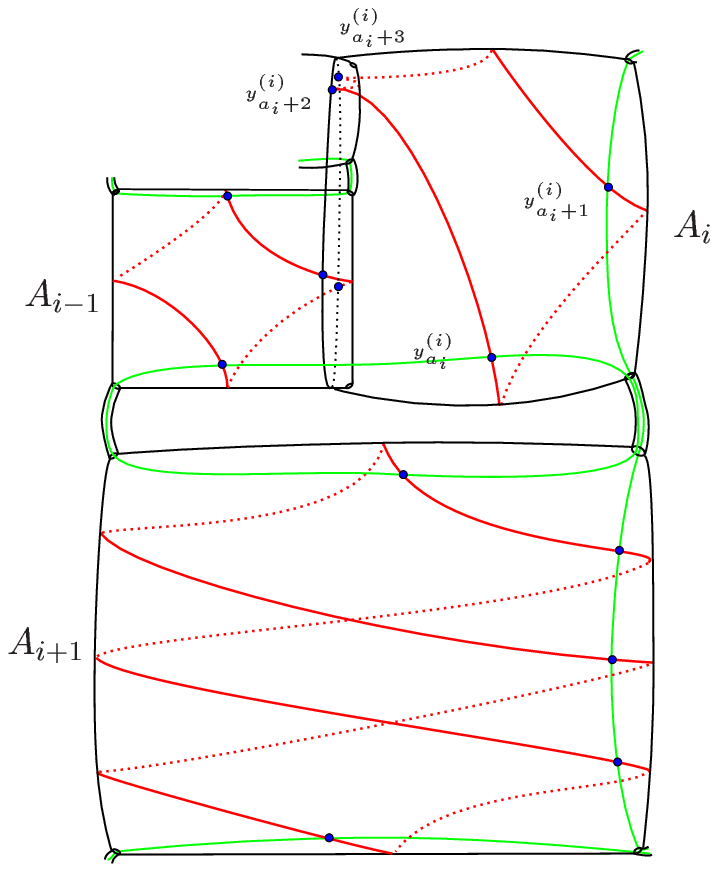}\label{fig78}
  \end{center}
   \end{minipage}
 \begin{minipage}{0.49\hsize}
  \begin{center}
   \includegraphics[scale=0.9]{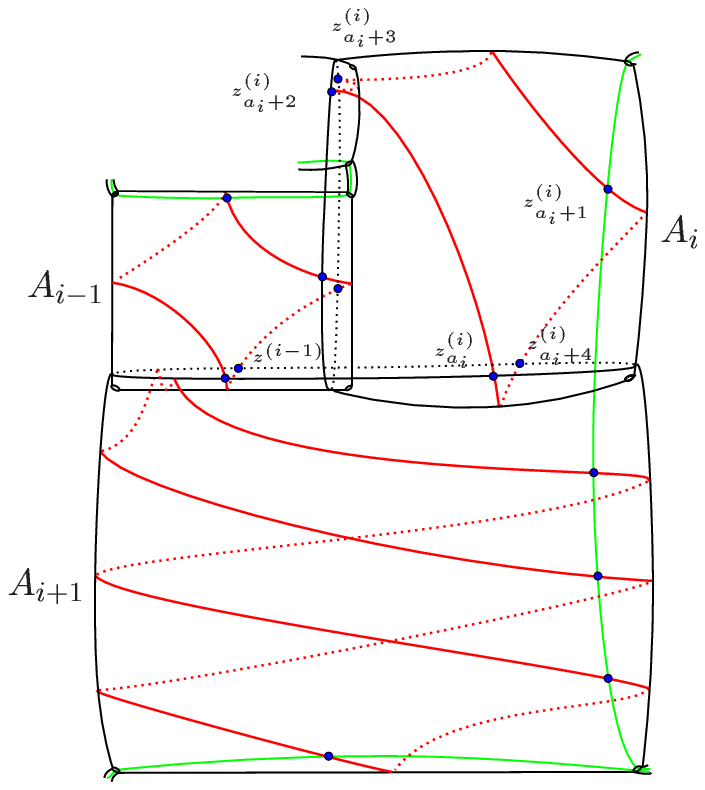}\label{fig79}
  \end{center}
  \end{minipage}
  \caption{True vertices in the spine $P_j$ (shown in the left hand side) and in the spine $P_{j+1}$ (shown in the right hand side) in the case of $a_{i-1}=1$.}
  \label{fig79}
\end{figure}

\end{itemize}

By the above arguments (i) and (ii), the number of true vertices in each pillowcase changes as
\[
\begin{cases}
x_{i-1}^{(j+1)}=x_{i-1}^{(j)}+1 \\
x_i^{(j+1)}=x_i^{(j)}-1 \\
x_{i+1}^{(j+1)}=x_{i+1}^{(j)}-1 \\
x_k^{(j+1)}=x_k^{(j)} \;\;\;( k\neq i-1,i,i+1 ).
\end{cases}
\]
Therefore, we get
\[
\sum_{k=1}^n x_k^{(j+1)} =\sum_{k=1}^n x_k^{(j)}-1.
\]
Since $P'_j$ has $\sum_{k=1}^n x_k^{(j)} $ true vertices, the replacement $P'_j\to P'_{j+1}$ decrease the number of true vertices in $P'_{j+1}$ by one. Hence, by the inductive sequence $P'_0 \to P'_1 \to \cdots \to P'_r=P'$, the number of true vertices decreases by $r$. Now we apply Lemma \ref{lem1}. The number of true vertices in the spine $P'$ is
\[
\begin{split}
&\displaystyle\sum_{i=1}^n a_i +2(n-3)-r \\
=&\displaystyle\sum_{i=1}^n a_i +2(n-3)-\sharp \{ a_i=1\}.
\end{split}
\]
This completes the proof.
\end{proof}

\begin{exmp}\label{ex}

The spine $P_r$ of $S^3\setminus {\rm int}N(C(3,2,1,3,3))$ is as shown in Figure \ref{fig12}. The spine $P'$ is obtained by collapsing this from the boundaries.
\end{exmp}

\begin{figure}
\begin{center}

\includegraphics[scale=1.0]{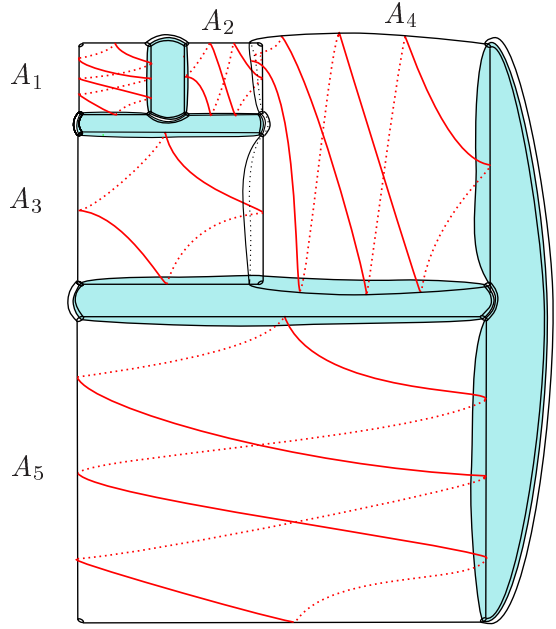}
\caption{The spine $P_r$ of $S^3\setminus {\rm int}N(C(3,2,1,3,3))$.}
\label{fig12}
\end{center}
\end{figure}

\begin{rmk}
Let $P(a_1,\ldots ,a_n)$ be a pretzel link, where $|a_i|>0$, $|a_1|,\ |a_n|>1$. We can construct a spine $P$ of the complement of $P(a_1,\ldots ,a_n)$, by attaching $2(n-1)+1$ tubes and $n$ disks to the $n$ tangles with $a_i$-twists, $i=1,\ldots ,n$ by a way similar to what we did in the proof of Lemma \ref{lem1}. Then the number of true vertices of the spine obtained from $P$ by collapsing becomes 
\[
|a_1|+ 2\sum_{i=1}^n |a_i| + |a_n|+n-4,
\]
which is an upper bound of the complexity $c(S^3\setminus P(a_1,\ldots , a_n))$. See \cite{Nemoto} for precise discussion.
\end{rmk}

\section{Proof of Corollaries \ref{cor2} and \ref{cor3}}

To prove Corollary \ref{cor2}, we need to have a lower bound on the complexity of $S^3\setminus C(a_1,\ldots ,a_n)$. Let ${\rm vol}(M)$ denote the hyperbolic volume of a hyperbolic 3-manifold $M$, and $v_3$ denote the hyperbolic volume of the regular ideal tetrahedron in the hyperbolic 3-space $\mathbb{H}^3$, that is, $v_3=1.01494\ldots$.

In order to prove Corollary \ref{cor2}, we will use the following theorems.

\begin{thm}[\cite{PV}]\label{thmPV}
Let $K(p,q)$ be a hyperbolic two-bridge link with $p/q=[a_1,\ldots ,a_n]$, $a_i>0$ and $a_1,\ a_n>1$. Then,
 
\begin{equation}\label{equ1}
{\rm vol}(S^3 \setminus {\rm int}N(K(p,q)))\geqslant v_3\cdot \max \{ 2,2n-2.6667 \dots \}.
\end{equation}
\end{thm}

A spine is called {\it simple} if the link of each point is either a circle, a theta-graph, or $K_4$, and it is called {\it special} if each 2-strata of the simple spine is an open disk. Remark that if $P$ is a special spine of a link complement $M$, then its dual is a topological ideal triangulation of $M$, and vice versa. 

\begin{thm}[\cite{Matveev}]\label{thmM}
Let $M$ be a compact, irreducible and boundary-irreducible 3-manifold which differs from a 3-ball, $S^3$, $\RP^3$, $L(3,1)$ and suppose that all proper annuli in $M$ are inessential. Then, for any almost-simple spine of $M$, there exists a special spine of $M$ which has the same or a fewer number of true vertices.
\end{thm}

Now we give a few notations. Suppose that $K(p,q)$ is hypebolic. Let $\mathcal{T}$ denote a topological ideal triangulation of $S^3\setminus {\rm int}N(K(p,q))$, $n(\mathcal{T})$ denote the number of ideal tetrahedra of $\mathcal{T}$ and $\sigma_\text{ideal}(S^3 \setminus {\rm int}N(K(p,q)))$ denote the minimal number of $n(\mathcal{T})$. By inequality (\ref{equ1}) we have
\begin{equation}\label{equ3}
\begin{split}
\sigma_{\rm ideal}(S^3 \setminus {\rm int}N(K(p,q))) &\geqslant \frac{\text{vol}(S^3 \setminus {\rm int}N(K(p,q)))}{v_3} \\
&\geqslant \max \{ 2,2n-2.6667 \dots \}.
\end{split}
\end{equation}

The next proposition suggests that we can replace the left hand side of inequality (\ref{equ3}) by the complexity $c(S^3\setminus {\rm int}N(K(p,q)))$. 

\begin{prop}\label{prop1}
If $L$ is a hyperbolic link then, 
\[ 
	\sigma_\mathrm{ideal}(S^3 \setminus \mathrm{int}N(L)) = c(S^3 \setminus \mathrm{int}N(L)).
\]
\end{prop}
\begin{proof}
Let us prove the inequality ($\leqslant$). Since $S^3\setminus \mathrm{int}N(L)$ is hyperbolic, it is irreducible and boundary-irreducible, and contains no essential annuli. By Theorem \ref{thmM}, we can deform any almost-simple spines of $S^3\setminus \mathrm{int}N(L)$ into a special one such that it has the same or a fewer number of true vertices. Since the dual of a special spine is a topological ideal triangulation of $S^3\setminus \mathrm{int}N(L)$, we have $\sigma_\text{ideal}(S^3 \setminus \mathrm{int}N(L)) \leqslant c(S^3 \setminus \mathrm{int}N(L))$.
The inverse inequality ($\geqslant$) is obvious, since \{ Special spine \} $\subset$ \{ Almost-simple spine \}.
\end{proof}  

From Theorem \ref{thm1}, Theorem \ref{thmPV} and Proposition \ref{prop1}, we can determine the exact values of the complexities for an infinite sequence of two-bridge links as mentioned in Corollary \ref{cor2}.

\begin{proof}[Proof of Corollary \ref{cor2}]
Let $K(p,q)$ be hyperbolic. By Theorem \ref{thmPV} and Proposition \ref{prop1}, the following inequality holds:
\[
\begin{split}
2n-2.66 \leqslant \sigma_\text{ideal}(S^3 \setminus \mathrm{int}N(K(p,q)))&=c(S^3 \setminus \mathrm{int}N(K(p,q))) \\
&\leqslant \displaystyle\sum_{i=1}^{n}a_i+2(n-3)-\sharp \{ a_i=1\}. 
\end{split}
\]
In particular, if $K(p,q)=C(2,1,\ldots ,1,2)$ then 
\[
\begin{split}
2n-2.66 \leqslant \sigma_\text{ideal}(S^3 \setminus \mathrm{int}N(K(p,q))) &= c(S^3 \setminus \mathrm{int}N(K(p,q))) \\
&\leqslant 2n-2.
\end{split}
\]
\end{proof}

\begin{rmk}\label{rmk2}
Sakuma and Weeks constructed canonical decompositions of hyperbolic two-bridge link complements explicitly in \cite{SW}. 
Calculating the number of ideal tetrahedra in their ideal triangulation, we get the upper bound

\begin{equation}\label{eq5}
\sigma_\mathrm{ideal}(S^3\setminus \mathrm{int}N(K(p,q)))
\leqslant 2\sum_{i=1}^n a_i -6.
\end{equation}

Let $c$ be the number of true vertices of the spine constructed in Theorem \ref{thm1}. We can obtain a special spine with the same as or a fewer number of vertices than $c$ by applying Theorem \ref{thmM}. Hence by considering its dual, we can obtain a topological ideal triangulation of $S^3\setminus K(p,q)$ consisting of at most $c$ ideal tetrahedra. If $p/q=[2,1,\ldots ,1,2]$ then the upper bound $c$ of the number of ideal tetrahedra constructed in Theorem \ref{thm1} coincides with the upper bound in inequality (\ref{eq5}). In general, the upper bounds obtained by our construction are better than those obtained in \cite{SW}.

\end{rmk}

Finally we give a proof of Corollary \ref{cor3}.

\begin{proof}[Proof of Corollary \ref{cor3}]

Since $a_1>1$, there exists a tube connecting $A_1$ and $A_2$
such that the union of the meridian-disk $D$ of this tube
and the spine $P'$ constructed in the proof of Theorem \ref{thm1}
has only one true vertex on the boundary of $D$.
Let $P_d$ be a spine of the $d$-fold cyclic covering space $\widetilde{M}$ of $S^3\setminus K(p,q)$ induced by $P'$.

Suppose that $p$ is odd, that is, $K(p,q)$ is a knot. 
Recall that the complexity of a closed 3-manifold is by definition the complexity of that manifold minus an open ball. Therefore, the complexity of $M_d(K(p,q))$ is at most the number of true vertices of the spine obtained from $P_d$ by attaching a meridian-disk along the preimage of the boundary of $D$. Thus we have

\[
c(M_d(K(p,q)))\leqslant d\Bigl( \sum_{i=1}^n a_i +2(n-3)-\sharp \{a_i=1\}\Bigr)+d.
\]

Suppose that $p$ is even, that is, $K(p,q)$ is a link. 
We attach  one more meridian-disk $D'$ to the other boundary component
of $\widetilde M$ such that the union of $P_d$, $D$ and $D'$ has two true
vertices on the boundary of $D'$.
It is known in \cite{Matveev} that the complexity does not change even if we remove several open balls. 
Therefore, the complexity of $M_d(K(p,q))$ is bounded above by
the number of true vertices of this union. Thus we have

\[
c(M_d(K(p,q)))\leqslant d\Bigl( \sum_{i=1}^n a_i +2(n-3)-\sharp \{ a_i=1\}\Bigr)+3d.
\] 

\noindent This completes the proof.
\end{proof}

\begin{rmk}\label{rmk1}
The complexities of $3$-manifolds obtained as meridian-cyclic branched coverings along two-bridge links had been studied in \cite{PV}. We can easily check that the upper bound in Corollary \ref{cor3} is better than theirs.
\end{rmk}

\vspace{10mm}

\noindent{\sc Mathematical Institute, Tohoku University, Sendai, 980-8578, Japan}\\
\email{{\it E-mail address}: ishikawa@math.tohoku.ac.jp}

\vspace{2.5mm}

\noindent\address{{\sc Mathematical Institute, Tohoku University, Sendai, 980-8578, Japan}}\\
\email{{\it E-mail address}: keisuke.nemo0121@gmail.com}

\end{document}